\documentclass[a4paper,reqno,11pt]{amsart}
\usepackage[a4paper,margin=1in]{geometry}
\usepackage{graphicx}
\usepackage{amsmath,amssymb,amsthm}
\usepackage{hyperref}
\usepackage{xcolor}
\usepackage{caption}
\usepackage{subcaption}
\theoremstyle{plain}
\newtheorem{theorem}{Theorem}[section]
\newtheorem{lemma}[theorem]{Lemma}
\newtheorem{proposition}[theorem]{Proposition}

\theoremstyle{definition}

\numberwithin{equation}{section}

\newcommand{\R}{\mathbb{R}}
\newcommand{\barU}{\bar{U}}
\newcommand{\pitwo}{\tfrac{\pi}{2}}
\newcommand{\twth}{\tfrac{2}{3}}
\newcommand{\thtw}{\tfrac{3}{2}}
\newcommand{\abs}[2][]{#1\lvert #2 #1\rvert}

\title{An asymptotic behaviour near the crest of waves of extreme form on water of finite depth}

\author{Vladimir Kozlov$^1$, Evgeniy Lokharu$^1$}
\address{$^1$Department of Mathematics, Link\"oping University, SE-581 83 Link\"oping, Sweden}

\begin{document}
	
\begin{abstract}

We prove local higher-order asymptotics for extreme water waves with vorticity near stagnation points. We obtain that the behaviour of solutions and their regularity depend substantially on the vorticity. In particular, we show that extreme waves with a negative vorticity distribution have concave profiles near the crest. Our approach is based on new regularity results and asymptotic analysis of the corresponding nonlinear problem in a half-strip. Our main result is local and therefore is valid for a broad range of problems, such as for waves with a piecewise constant vorticity, stratified waves, flows with counter-currents or waves on infinite depth. 	
\end{abstract}

\maketitle

\section{Introduction}

Extreme waves or also known as waves of greatest height is an important phenomena in the mathematical theory of water waves. The story goes back to Sir George Stokes \cite{Stokes49} who in 1880s studied periodic solutions of the water wave problem when the wavelength is fixed. He assumed that such waves can be parametrized by the wave height $\sup_{x \in \R} \eta(x) - \inf_{x \in \R} \eta(x)$, where $y = \eta(x)$ is the surface profile. In \cite{Stokes80} Stokes conjectured that the family of periodic waves contains ``\textit{the wave of greatest height}'' distinguished by sharp crests of included angle $120^\circ$, see Figure 1. That was a remarkable hypothesis for the time and it took about two centuries before it was rigorously justified.

The Stokes conjecture might be divided into two independent parts:
\begin{itemize}
	\item[(i)] there exists a sufficiently regular travelling wave solution of the water wave problem that enjoys stagnation at every crest (where the horizontal and vertical components of the relative velocity fields vanish);
	\item[(ii)] every solution from (i) with surface profile $\eta$ must satisfy
	\[
	\lim_{x \to x_0\pm} \eta_x(x) = \mp \frac{1}{\sqrt{3}}
	\]
	at every stagnation point $(x_0,\eta(x_0))$; this corresponds to the included angle $120^\circ$.
\end{itemize}
The main difficulty about the Stokes conjecture is that waves with surface stagnations in (i) are normally large-amplitude solutions with low regularity. That makes their analysis complicated and requires different tools compared to perturbation methods for small-amplitude nonlinear water waves. The first construction of large-amplitude waves is due to Krasovskii in \cite{Krasovskii60}, who proved the following statement about irrotational water waves in deep water. It was shown that for any given flux $Q$, wavelength $L$ and $\beta \in (0,\tfrac{\pi}{6})$ there exists a Stokes wave with $\max \theta = \beta$, where $\theta$ is the inclination angle to the horizontal. Even so this is a beautiful statement with a clear geometrical interpretation it does not explain if such waves are close to stagnation or not. About two decades later Keady and Norbury \cite{KeadyNorbury78} used a different approach based on the global bifurcation theory for the Nekrasov equation. They proved that there exist Stokes waves (symmetric periodic profiles having exactly one crest and trough in every minimal period) that are arbitrary close to the stagnation. The limiting wave with surface stagnation was obtained by Toland \cite{Toland1978a} in the case of infinite depth. For the corresponding results about extreme (by extreme waves we mean solutions with surface stagnation points as in (i)) Stokes and solitary waves on finite depth we refer to Amick and Toland \cite{AmickToland81a,AmickToland81b}.

The second part (ii) of the Stokes conjecture was verified independently by Amick, Fraenkel, and Toland in \cite{AmickFraenkelToland82} and by Plotnikov \cite{Plotnikov82}. Much later, Plotnikov and Toland \cite{Plotnikov2004} proved the existence of extreme periodic waves that are convex everywhere outside crests. The Stokes conjecture in the irrotational case was refined by Varvaruca and Weiss in \cite{Varvaruca2011}, who proved (ii) for solutions under weak regularity assumptions and without any symmetry or monotonicity constraints. In particular, (ii) turned out to be a local property and is valid for the extreme solitary wave found in \cite{AmickToland81b}. 

So far all mentioned results concerned with water waves on the surface of irrotational flows, while the rotational theory is much less developed. The subject has attracted a significant interest with the pioneering study by Constantin and Strauss \cite{ConstantinStrauss04}. The authors used bifurcation and degree theories to construct global connected sets of large-amplitude periodic waves with vorticity that can be arbitrary close to the stagnation. Thus, one could think of constructing an extreme wave by passing to the limit along a sequence of waves approaching the stagnation. This was formally done in \cite{Varvaruca09} under certain assumptions on the vorticity. We say formally because it is not known if the limiting wave is trivial or not. By a trivial extreme wave we mean a laminar flow whose surface or bottom consists of stagnation points. To overcome this difficulty a different approach was proposed in \cite{LokharuKoz2020} and extreme waves subject to (i) were constructed.

The second part (ii) of the Stokes conjecture is more complicated for waves with vorticity. In their study \cite{Varvaruca2012} Varvaruca and Weiss found (without proving the existence) that surface profiles near stagnation points are either (a1) Stokes corners ($120^\circ$), (a2) horizontally flat with $\eta_x(0)=0$, or (a3) overhanging horizontal cusps. So far it is not known if options (a2) and (a3) are possible, though (a2) is always true for extreme laminar flows. 

Beside the questions (i) and (ii) one can also ask about the regularity of extreme waves near stagnation points. The only results of this type are \cite{Amick1987a} and \cite{McLeod1987} for irrotational waves on  infinite depth. An asymptotic expansion for the inclination angle in the conformal variables was obtained in \cite{Amick1987a}, while \cite{McLeod1987} contains an analysis of the leading order coefficient. In the present paper we investigate these questions about extreme water waves with vorticity. Assuming (ii) we obtain local asymptotics for the surface profile in the right neighbourhood of the stagnation point. The latter asymptotics essentially depend on values of the vorticity function near the surface. As a consequence we obtain a surprising result: the surface profile can be concave near the stagnation point when the vorticity has the right sign. This observation is confirmed by several numerical studies, such as \cite{JOY2008} and \cite{Dyachenko2019a}. In some sense our results improve the statement of \cite{Amick1987a} for waves on infinite depth, which is obtained only in conformal variables, while we provide expansions in physical variables. 

\begin{figure}[t!]
	\centering
	\begin{subfigure}[t]{0.5\textwidth}
		\centering
		\includegraphics[scale=0.8]{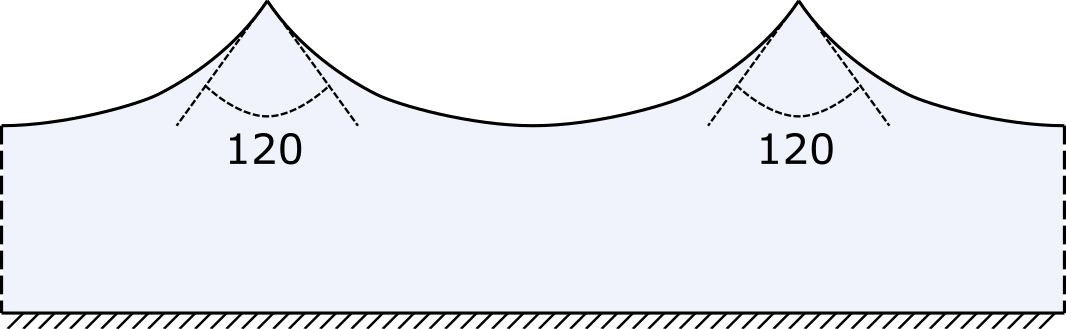}
		\caption{an extreme Stokes wave}
	\end{subfigure}%
	~ 
	\begin{subfigure}[t]{0.5\textwidth}
		\centering
		\includegraphics[scale=0.8]{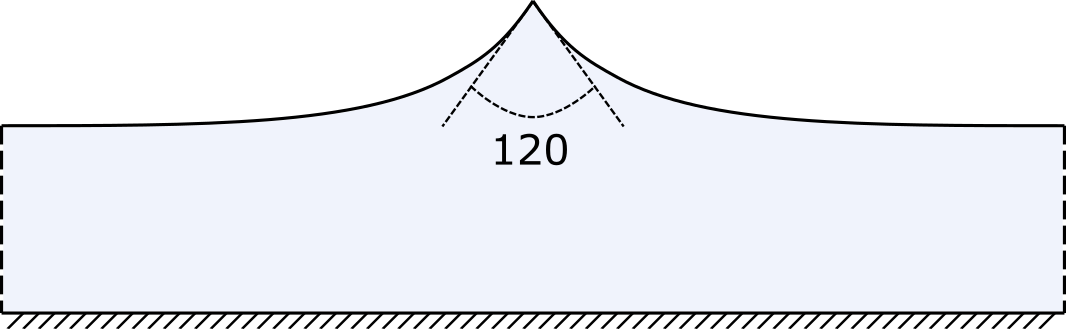} \caption{an extreme solitary wave}
	\end{subfigure}
	\caption{}
\end{figure}

\section{Statement of the problem}

We consider the classical water wave problem for two-dimensional steady waves with vorticity on water of finite depth. An infinite fluid region is occupied with an ideal fluid of constant unit density, separated from the air by an unknown free surface, where the effects of surface tension and air motion are neglected. Assuming the motion of the fluid is steady we can find an appropriate coordinate system moving with a constant speed $c > 0$ in which the flow is stationary and governed by Euler equations:
\begin{subequations}\label{eqn:trav}
	\begin{align}
		\label{eqn:u}
		(u-c)u_x + vu_y & = -P_x,   \\
		\label{eqn:v}
		(u-c)v_x + vv_y & = -P_y-g, \\
		\label{eqn:incomp}
		u_x + v_y &= 0, 
	\end{align}
	which hold true in a two-dimensional fluid domain $D_\eta = \{ (x,y) \in \R^2: 	0 < y < \eta(x), \ x \in \R \}$. Here $(u,v)$ are components of the relative velocity field, $y = \eta(x)$ is the surface profile, $c$ is the wave speed, $P$ is the pressure and $g$ is the gravitational constant. The corresponding boundary conditions are
	\begin{alignat}{2}
		\label{eqn:kinbot}
		v &= 0&\qquad& \text{on } y=0,\\
		\label{eqn:kintop}
		v &= (u-c)\eta_x && \text{on } y=\eta,\\
		\label{eqn:dyn}
		P &= P_{\mathrm{atm}} && \text{on } y=\eta.
	\end{alignat}
\end{subequations}
It is often assumed in the literature that the flow is irrotational, that is $v_x - u_y$ is zero everywhere in the fluid domain. Under this assumption components of the velocity field are harmonic functions, which allows to apply methods of complex analysis. Being a convenient simplification it forbids modeling of non-uniform currents, commonly occurring in nature. In the present paper we will consider rotational flows, where the vorticity function is defined by
\begin{equation} \label{vort}
	\omega = v_x - u_y.
\end{equation}
Throughout the paper we assume that the flow is free from stagnation points and the horizontal component of the relative velocity field does not change sign, that is
\begin{equation} \label{uni}
	u-c < 0 
\end{equation}
everywhere in the fluid. We call such flows unidirectional.

In the two-dimensional setup relation \eqref{eqn:incomp} allows to reformulate the problem in terms of a stream function $\psi$, defined implicitly by relations
\[
\psi_y = c-u, \ \ \psi_x =  v.
\]
This determines $\psi$ up to an additive constant, while relations \eqref{eqn:kinbot},\eqref{eqn:kinbot} force $\psi$ to be constant along the boundaries. Thus, by subtracting a suitable constant, we can always assume that
\[
\psi = m, \ \ y = \eta; \ \ \psi = 0, \ \ y = 0.
\]
Here $m$ is the mass flux, defined by
\[
m = \int_0^\eta (u-c) dy.
\]
In what follows we will use non-dimensional variables proposed by Keady \& Norbury \cite{KeadyNorbury78}, where lengths and velocities are scaled by $(m^2/g)^{1/3}$ and $(mg)^{1/3}$ respectively; in new units $m=1$ and $g=1$. For simplicity we keep the same notations for $\eta$ and $\psi$.

Taking the curl of Euler equations \eqref{eqn:u}-\eqref{eqn:incomp} one checks that the vorticity function $\omega$ defined by \eqref{vort} is constant along paths tangent everywhere to the relative velocity field $(u-c,v)$; see \cite{Constantin11b} for more details. Having the same property by the definition, stream function $\psi$ is strictly monotone by \eqref{uni} on every vertical interval inside the fluid region. These observations together show that $\omega$ depends only on values of the stream function, that is
\[
\omega = \omega(\psi).
\]
This property and Bernoulli's law allow to express the pressure $P$ as
\begin{align}
	\label{eqn:bernoulli}
	P-P_\mathrm{atm} + \frac 12\abs{\nabla\psi}^2 + y  + \Omega(\psi) - \Omega(1) = const,
\end{align}
where 
\begin{align*}
	\Omega(\psi) = \int_0^\psi \omega(p)\,dp
\end{align*}
is a primitive of the vorticity function $\omega(\psi)$. Thus, we can eliminate the pressure from equations and obtain the following problem:

	\begin{subequations}\label{sysstream}
	\begin{alignat}{2}
	\label{sys:stream:lap}
	\Delta\psi + \omega(\psi)&=0 &\qquad& \text{for } 0 < y < \eta,\\
	\label{sys:stream:bern}
	\tfrac 12\abs{\nabla\psi}^2 +  y  &= r &\quad& \text{on }y=\eta,\\
	\label{sys:stream:kintop} 
	\psi  &= 1 &\quad& \text{on }y=\eta,\\
	\label{sys:stream:kinbot} 
	\psi  &= 0 &\quad& \text{on }y=0.
	\end{alignat}
At the crest located at $(0,r)$ the flow has a stagnation point, that is
	\begin{equation}\label{stag}
		\psi_y(0,r)=\psi_x(0,r) = 0, \ \ \eta(0) = r.
	\end{equation}

In what follows we will consider a local solution of the problem, which unidirectional in a neighbourhood of the stagnation point. Thus, the boundary condition \eqref{sys:stream:kinbot} might be omitted. Therefore, our results will be valid for waves on infinite depth and for flows with counter-currents.

We assume that the surface profile $\eta$ is defined on an open interval $I=(-\delta,\delta)$ and is an even function there. Furthermore, we assume that $\eta_x \in C^{1}([0,\delta])$, while
\begin{equation} \label{eta120}
\lim_{x \to 0+} \eta_x(x) = - \tfrac{1}{\sqrt{3}}.
\end{equation}
\end{subequations}
The stream function $\psi$ is defined over the set
\[
	D_\eta^\delta = \{ (x,y) \in \R^2: \ x^2 + (y-r)^2 \leq \delta^2, \ \ -\delta < x < +\delta, \ \ y \leq \eta(x) \}.
\]
As for regularity of $\psi$ we assume that $\psi \in C^{1}(\overline{D_\eta^\delta}) \cap C^2(D_\eta^\delta)$. We additionally require that $\psi$ is even in the $x$-variable.

\begin{theorem} \label{thm1}
	Let $(\psi,\eta)$ be as above and solve \eqref{sysstream} in $D_\eta^\delta$ for some $\delta>0$. Assume that $\omega \in C^1([1-\delta,1])$ and $\psi_y > 0$ for $(x,y) \in D_\eta^\delta, y \neq r$. Then
	\begin{equation}\label{eta1}
		\eta_x(x) = - \tfrac{1}{\sqrt{3}} + \frac{2^{\tfrac32}}{3^{\tfrac54}} \omega(1) \sqrt{x} + a_1 \omega^2(1) x+f(x)
	\end{equation}
	for some explicit $a_1 > 0$, 	where $f = O(x^{\tfrac32(\tau_1-1)}), \ f_x = O(x^{\tfrac32(\tau_1-1)-1})$ as $x \to 0+$ and $\tau_1 \approx 1.8$ is the smallest root of $\tau_1 = - \tfrac{1}{\sqrt{3}} \cot(\tfrac{\pi}{2} \tau_1)$. 
\end{theorem}

It follows from the theorem that $\eta \in C^{2,\gamma}([0,\delta])$ for some small $\gamma > 0$, provided $\omega(1) = 0$. Beside \eqref{eta1} one can also obtain the corresponding asymptotics for the stream function, where the leading order term is determined by the Stokes corner flow. For more details see Section \ref{secasymppsi}. \\

Our proof of Theorem \ref{thm1} is given in the next sections. 

\section{H\"older regularity}

Our aim in this section is to obtain the optimal global regularity of an extreme wave. 

\begin{theorem} \label{thmreg}
	Suppose that $(\psi,\eta)$ satisfies assumptions of Theorem \ref{thm1}. Then there exists $0 < \delta_1 < \delta$ such that 
	\[
	\|\nabla \psi\|_{C^{\tfrac12}(\overline{\Omega}_\eta^{\delta_1})} \leq C,
	\]
	where the constant $C$ depends only on $\delta, \delta_1, \omega$ and $r$.
\end{theorem}

Note that the H\"older exponent $1/2$ is optimal and can not be improved. Thus, the H\"older space $C^{1,1/2}$ is the natural choice for the stream function $\psi$ of an extreme wave. Our proof consists of several lemmas. The key property of the stream function is stated in the next lemma. 

\begin{lemma} \label{l:psiybot}
	There exist $\delta_1 > 0$ and constants $C_1,C_2>0$ such that 
	\begin{equation} \label{psiybot}
	C_1 \rho^{\tfrac12} < \psi_y(x,y) < C_2 \rho^{\tfrac12}
	\end{equation}
	for all $(x,y) \in \Omega_\eta^{\delta_1}$, where $\rho = (x^2 + (r-y)^2)^{\tfrac12}$ and $C_1,C_2$ are independent of $\rho$.
\end{lemma}
\begin{proof}
	Note that \eqref{psiybot} holds true along the surface in $\Omega_\eta^{\delta}$, which is a direct consequence of the Bernoulli equation \eqref{sys:stream:bern} and \eqref{eta120}. Now let $0 < x < \delta_1$ be given, where $\delta_1 > 0$ is such that
	\[
		\tfrac1{2\sqrt{3}} < |\eta_x| < \tfrac2{\sqrt{3}}
	\]
	for all $|x| < \delta_1$. Then we put $R = \tfrac{1}{10}x$ and apply Theorem 8.26 in \cite{GilbargTrudinger01} in $B_{\rho}(x,\eta(x))$ for the function $\psi_y$ solving
	\[
		\Delta \psi_y + \omega'(\psi) \psi_y = 0.
	\]
	This shows that the left inequality in \eqref{psiybot} is valid in $B_{R}(x,\eta(x))$ for any $|x| < \delta_1$. For an interior estimate one can use the Harnack principle from Theorem 8.20 in \cite{GilbargTrudinger01}. It remains to prove the upper bound for $\psi_y$. For that purpose we consider the function
	\[
		f= \psi_y^2 - A (r-y) + B (r-y)^2,
	\]
	where 
	\[
	B = \sup_{\Omega_\eta^{\delta_1}} |\psi_y|^2 |\omega'(\psi)|, \ \ \ A = \sup_{\partial \Omega_\eta^{\delta_1}} \psi_y^2 (r-y)^{-1} + B(r-y).
	\]
	The choice of $B$ implies that $\Delta f \geq 2 |\nabla \psi_y|^2 \geq 0$ so that $f$ attains its maximum at the boundary of $\Omega_\eta^{\delta_1}$, where $f$ is nonpositive by the choice of $A$. Note that the constant $A$ depends on $\delta_1$ and might be large though it is finite since \eqref{psiybot} is valid along the upper boundary. Thus, by the maximum principle we have $f \leq 0$ in $\Omega_\eta^{\delta_1}$ so that 
	\[
		\psi_y^2 \leq A y \leq C A \rho \ \ \text{in} \ \ \Omega_\eta^{\delta_1}.
	\]
	This finished the proof of the lemma.
\end{proof}
\begin{figure}[t!]
	\centering
	\includegraphics[scale=0.8]{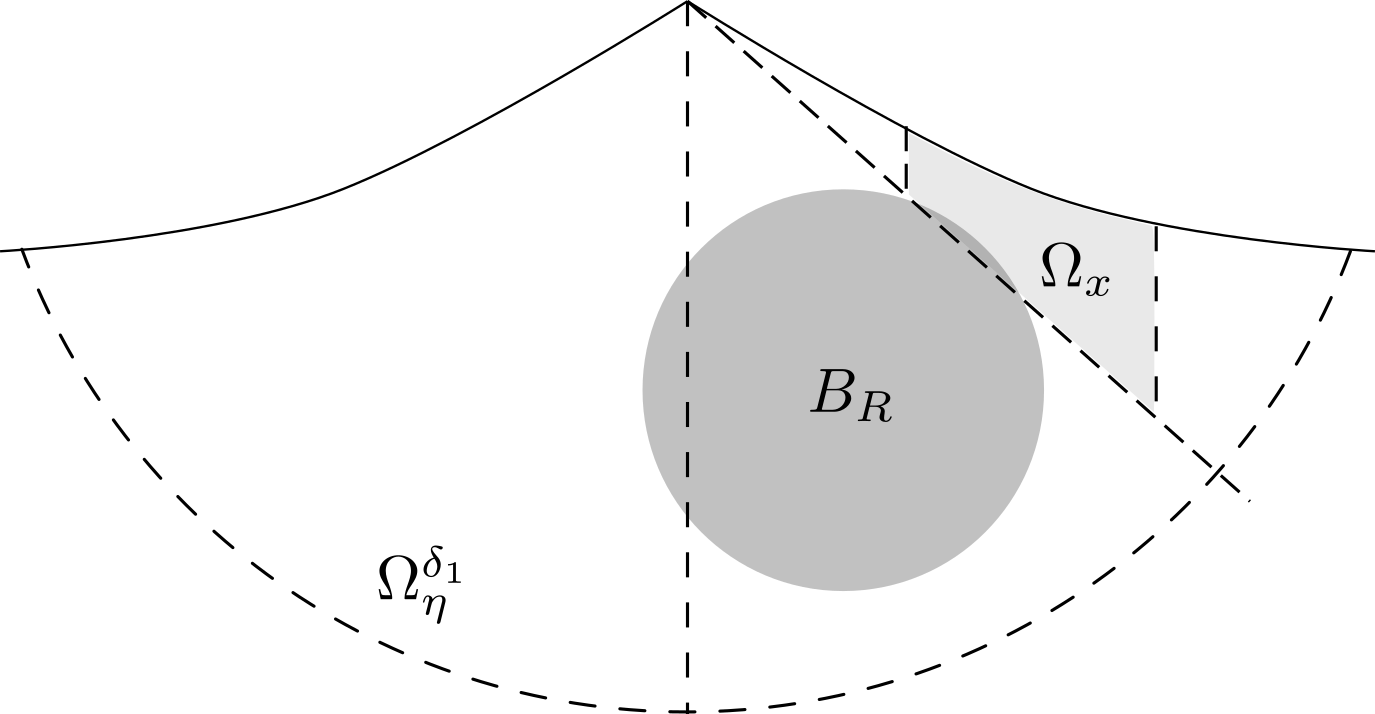} \caption{Domains $\Omega_x$ and $B_R$}
	\label{fig:polar}
\end{figure}
Now using Lemma \ref{l:psiybot} we can obtain local estimates in the domain
\[
\Omega_x = \{ (x',y'): \in \Omega_\eta^{ \delta_1}: \ \ \tfrac12 x < x' < \tfrac32 x, \ \ y' > r- \tfrac{4}{\sqrt{3}}x' \},
\]
where $\delta_1$ is the constant from Lemma \ref{l:psiybot} and $0 < x < \tfrac23 \delta_1$ (see Figure 1). Note that
\begin{equation} \label{xrho}
	C_1 x  < \rho < C_2 x
\end{equation}
for all $(x,y) \in \Omega_{x}$ with some absolute constants $C_1,C_2>0$.

\begin{lemma} \label{lemmaboundary}
	There exist $C>0$ and $x_0 > 0$ such that
	\[ 
	|\nabla \psi (x_1,y_1)-\nabla \psi (x_2,y_2)| \leq C x^{-\tfrac12} (|x_1-x_2|+|y_1-y_2|)
	\]
	for all $0<x<x_0$ and for all $(x_1,y_1),(x_2,y_2) \in \Omega_x$. The constant $C$ is independent of $x$. Furthermore, we have
	\begin{equation}\label{c2gammapsi}
		\|\psi\|_{C^{2,\gamma}(\overline{\Omega_x})} \leq C x^{-\left(\tfrac12 + \gamma\right)}
	\end{equation}
	for $0<x<x_0$ and any $\gamma \in (0,1)$.
\end{lemma}
\begin{proof}
	The idea of the proof is to perform the partial hodograph transform in $\Omega_x$ obtaining a nonlinear equation with good scaling properties. More precisely, we put
	\[
	q = x, \ \ p = 1-\psi(x,y), \ \ h(q,p) = r-y.
	\]
	A direct computation shows that $h$ solves an elliptic equation
	\[
	\frac{1+h_q^2}{h_p^2} h_{pp}-2 \frac{h_q}{h_p} h_{qp}+h_{qq}+h_p \omega(1-p) = 0 \ \ \text{in} \ \ S_{x},
	\]
	where $S_x$ is the corresponding to $\Omega_x$ domain in the $(q,p)$-variables given by
	\[
	S_x = \{(q,p): \tfrac12 x < q < \tfrac32 x, \ \ C(q) x^{\tfrac32} < p < 0 \},
	\]
	where $C_1 < C(q) < C_2$ for some absolute constants $C_1,C_2 > 0$. Thus, we scale variables as follows:
	\[
	q = x \hat{q}, \ \ p = x^{\tfrac32} \hat{p}, \ \ h(q,p) = x \hat{h}(\hat{q},\hat{p}).
	\]
	Another computation gives
	\[
	h_p = x^{-\tfrac12} \hat{h}_{\hat{p}}, \ \ h_q = \hat{h}_{\hat{q}}, \ \ h_{pp} =x^{-2} \hat{h}_{\hat{p}\hat{p}}, \ \ h_{qq} = x^{-1} \hat{h}_{\hat{q}\hat{q}},
	\]
	so that $\hat{h}$ solves
	\[
	\frac{1+\hat{h}_{\hat{q}}^2}{\hat{h}_{\hat{p}}^2} \hat{h}_{\hat{p}\hat{p}}-2 \frac{\hat{h}_{\hat{q}}}{\hat{h}_{\hat{p}}} \hat{h}_{\hat{q}\hat{p}}+\hat{h}_{\hat{q}\hat{q}}+x^{\tfrac12} \hat{h}_{\hat{p}} \omega(1-x^{\tfrac32}\hat{p}) = 0 \ \ \text{in} \ \ \hat{S}_{x}.
	\]
	Note that 
	\[
	C_1 < \hat{h}, \hat{h}_{\hat{p}} < C_2 \ \ \text{in} \ \ \hat{S}_{x}
	\]
	for some constants $C_1,C_2 > 0$, which follows from Lemma \ref{l:psiybot} since $\hat{h}_p = x^{\tfrac12}\psi_y^{-1}$ and $x$ is comparable with $\rho$ by \eqref{xrho}. Along the upper boundary of $\hat{S}_{x}$ we have 
	\[
	\frac{1+\hat{h}_{\hat{q}}^2}{2\hat{h}_{\hat{p}}^2} - \hat{h} = 0,  \ \ \hat{p} = 0,
	\]
	which follows from the Bernoulli equation after the scaling. Thus, $\hat{h}$ solves a uniformly elliptic nonlinear boundary problem and from Theorem 1.1 \cite{Lieberman1986} we conclude that
	\begin{equation}\label{hhat}	
		\|\hat{h}\|_{C^2(\overline{\hat{S}_x})} \leq C
	\end{equation}
	for some constant $C>0$ independent of $x$ (we even get a higher regularity by it is not essential for our purposes). Now we can estimate
	\[
	\begin{split}
	|\psi_y(x_1,y_1)-\psi_y(x_2,y_2)| & \leq C x |h_p(q_1,p_1)-h_p(q_2,p_2)| \leq C x^{\tfrac12} |\hat{h}_{\hat{p}}(\hat{q}_1,\hat{p}_1)-\hat{h}_{\hat{p}}(\hat{q}_2,\hat{p}_2)| \\
	& \leq C x^{\tfrac12} (|\hat{q}_1-\hat{q}_2|+|\hat{p}_1-\hat{p}_2|) \leq C x^{-\tfrac12} (|x_1-x_2|+|y_1-y_2|)
	\end{split}		
	\]
	for all $(x_1,y_1),(x_2,y_2) \in \Omega_x$. Similarly, we obtain
	\[
	|\psi_x(x_1,y_1)-\psi_x(x_2,y_2)| \leq C x^{-\tfrac12} (|x_1-x_2|+|y_1-y_2|).
	\]
	The remaining inequality \eqref{c2gammapsi} can be obtained the same way from \eqref{hhat}. This finishes the proof.
\end{proof}

Now we consider an interior region $B_{R} \subset \Omega_\eta^\delta$, where $B_R$ is a ball of radius $R$ centred at $(0,y_R) \in \Omega_\eta^{\delta_1}$, where $r-y_R > \tfrac65 R$ (see Figure 1). Thus, the distance from $B_R$ to the upper boundary of $\Omega_\eta^\delta$ is of order $R$. From Lemma \ref{l:psiybot} we find that 
\begin{equation}\label{psiyB1}
C_1 R^{\tfrac32} < 1-\psi < C_2 R^{\tfrac32} \ \ \text{in} \ \ B_R.
\end{equation}
Next we perform the following scaling of variables in $B_R$:
\[
x = R \hat{x}, \ \ y = r - R \hat{y}, \ \ \psi(x,y) = 1-R^{\tfrac32} \hat{\psi}(\hat{x},\hat{y}).
\]
Thus, the new function $\hat{\psi}$ is defined in a ball $\hat{B}$ of radius $1$ and satisfy
\[
C_1 < \hat{\psi} < C_2 \ \ \text{in} \ \ \hat{B}
\]
where $C_1,C_2 > 0$ are constants independent of $R$, which follows from \eqref{psiyB1}. Now the classical elliptic theory yields that 
\[
\|\hat{\psi}\|_{C^{2,\gamma}(\hat{B})} \leq C.
\]
Scaling back, we obtain
\begin{equation} \label{gradint}
|\nabla \psi (x_1,y_1)-\nabla \psi (x_2,y_2)| \leq C R^{-\tfrac12} (|x_1-x_2|+|y_1-y_2|)
\end{equation}
for all $(x_1,y_1),(x_2,y_2) \in B_R$. Furthermore, we have
\begin{equation}\label{c2gammaball}
	\|\psi\|_{C^{2,\gamma}(B_R)} \leq C R^{-\tfrac12-\gamma}.
\end{equation}

Now we can complete the proof of Theorem \ref{thmreg} by establishing interior estimates. Let us consider two points $(x_1,y_1),(x_2,y_2) \in \Omega_\eta^{\delta_1}$. Let $\rho_j$ be the distance from $(x_j,y_j)$ to $(0,r)$ and let $\rho_0 = \min(\rho_1,\rho_2)$. Furthermore, we put $\rho_{12}$ to be the distance between $(x_1,y_1)$ and $(x_2,y_2)$. If $\rho_0 < \rho_{12}$, then we obtain
\[
|\psi_y (x_1,y_1)-\psi_y (x_2,y_2)| \leq 	|\psi_y (x_1,y_1)|+|\psi_y (x_2,y_2)| \leq C \rho_{12}^{\tfrac12}
\]
by \eqref{psiybot} as desired. Assume that $\rho_0 > \rho_{12}$. Then without loss of generality we can assume that both points $(x_1,y_1)$ and $(x_2,y_2)$ belong to one of the regions studied above: $B_{\rho_0}$ or $\Omega_{x}$, where $C_1 \rho_0 < |x| < C_2 \rho_0$. In both cases we conclude from \eqref{gradint} and Lemma \ref{lemmaboundary} that
\[
|\psi_y (x_1,y_1)-\psi_y (x_2,y_2)| \leq C \rho_0^{-\tfrac12} \rho_{12} \leq C \rho_{12}^{\tfrac12}.
\]
In a similar way one obtains an estimate for $\psi_x$.  The only difference is in the case when $\rho_0 < \rho_{12}$. We need to show that $|\psi_x(x_j,y_j)| < C \rho_j^{\tfrac12}$. But this follows from the previous interior and boundary estimates \eqref{gradint} and Lemma \ref{lemmaboundary} which imply that $|\psi_x| \leq C \psi_y$. This finishes the proof of the theorem.

\section{Reformulation of the problem}
\subsection{Logarithmic transformation}

\begin{figure}[t!]
	\centering
	\includegraphics[scale=0.8]{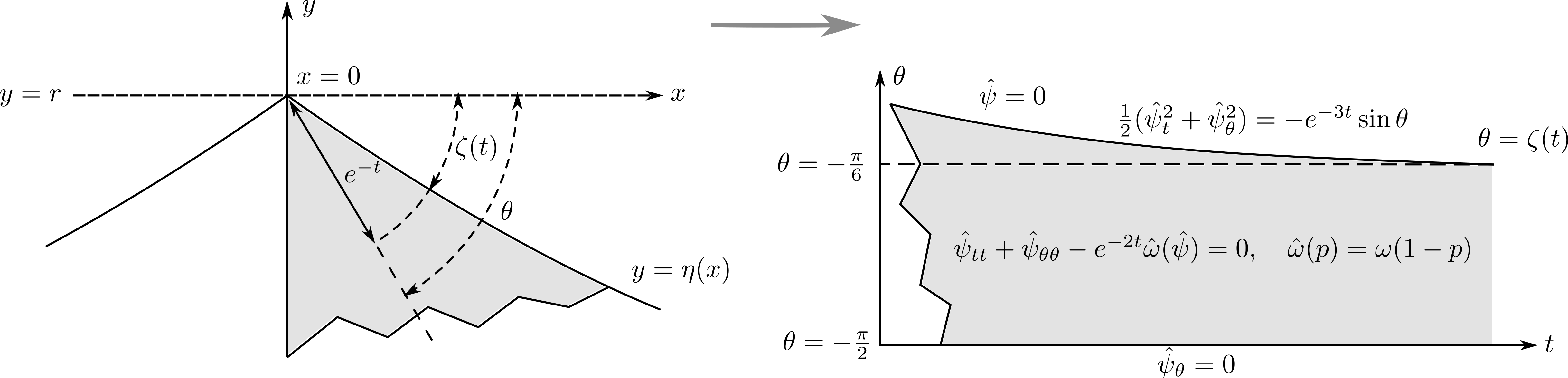} \caption{Transformation of the fluid domain near a stagnation point into a curved half-strip $\hat{S}$; the vertical $x=0$ becomes the bottom boundary $\theta=-\tfrac{\pi}{2}$ and the surface $y = \eta(x)$ turns into $\theta = \zeta(t)$.}
	\label{fig:polar}
\end{figure}

The problem \eqref{sysstream} in a corner can be transformed into the one in a strip through the following logarithmic transformation:
\[
x = e^{-t} \cos \theta, \ \ y = r+e^{-t} \sin \theta, \ \ t > -\ln \delta_1, \ \ -\tfrac{\pi}{2}<\theta < \zeta(t).
\]
Under this transformation the corresponding part of the fluid domain becomes a curved half-strip $\hat{S}$, infinite from the right; see Figure \ref{fig:polar}. The upper boundary of the "strip" corresponds to the surface profile, while the flat bottom is the image of the vertical line $x=0$. The new unknown function is
\[
\hat{\psi}(t,\theta) = 1-\psi(x,y), 
\]
which solves
\[
\hat{\psi}_{tt} + \hat{\psi}_{\theta\theta} - e^{-2t} \hat{\omega}(\hat{\psi}) = 0 \ \ \text{in} \ \ \hat{S}= \{ (t,\theta) \in \R^2:  t > -\ln\delta_1, \ \ -\pitwo < \theta < \zeta(t) \},
\]
where $\hat{\omega}(p) = \omega(1-p)$. The latter is a direct consequence of \eqref{sys:stream:lap}. According to our definitions we have
\begin{equation}\label{psi0}
	\hat{\psi} = 0 \ \ \text{on} \ \ \theta = \zeta,
\end{equation}
while the formula $\hat{\psi}_\theta = e^{-t} \sin\theta \psi_x - e^{-t} \cos\theta \psi_y$ shows that
\[
	\hat{\psi}_\theta = 0 \ \ \text{on} \ \ \theta = - \pitwo.
\]
On the surface $\theta = \zeta(t)$ we also have 
\begin{equation}\label{bern1}
\tfrac12 (\hat{\psi}_t^2 + \hat{\psi}_{\theta}^2) = -e^{-3t}\sin\theta \ \ \text{on} \ \ \theta = \zeta(t),
\end{equation}
which is a direct consequence of \eqref{sys:stream:kintop} and \eqref{sys:stream:bern}. Beside the boundary relations from above we also know from \eqref{eta120} that
\begin{equation}\label{zetalim}
\zeta(t) \to -\tfrac16 \pi \ \ \text{as} \ \ t \to +\infty.
\end{equation}
Furthermore, from \eqref{psi0} we find
\begin{equation} \label{xieta}
\zeta_t = - \frac{\hat{\psi}_t}{\hat{\psi}_{\theta}} = - \frac{\psi_x \cos{\zeta} + \psi_y \sin{\zeta}}{\psi_x \sin{\zeta} - \psi_y \cos{\zeta}} = - \frac{\eta_x \cos{\zeta}-\sin{\zeta}}{\eta_x \sin{\zeta} + \cos{\zeta}}.
\end{equation}
Thus, from \eqref{zetalim} and \eqref{eta120} we conclude that
\begin{equation} \label{xit}
\zeta_t \to 0 \ \ \text{as} \ \ t \to +\infty.
\end{equation}
Using \eqref{xieta} we can recover the surface profile $\eta$ as
\begin{equation}\label{extazeta}
	\eta_x = \frac{\sin\zeta - \zeta_t \cos\zeta}{\cos\zeta + \zeta_t \sin\zeta},
\end{equation}
which will be useful later.

A direct consequence of Theorem \ref{thmreg} is the following

\begin{lemma}\label{lemmahatpsi} For any $\gamma \in (0,1)$ there exists $C>0$ such that the inequality
	\[
		\|\hat{\psi}\|_{C^{2,\gamma}(\hat{S}_t)} \leq C e^{-\tfrac32 t},
	\]
is valid for all $t \geq -\ln{\delta_1}$, where $\hat{S}_t = \{(t',\theta) \in \overline{\hat{S}}: \ t \leq t' \leq t+1 \}$. 
\end{lemma}
\begin{proof}
	It follows from definitions that
	\begin{equation}\label{hatgradpsi}
		\hat{\psi}_\theta = e^{-t} \sin\theta \psi_x - e^{-t} \cos\theta \psi_y, \ \ \hat{\psi}_t = e^{-t} \cos\theta \psi_x +  e^{-t} \sin\theta \psi_y.
	\end{equation}
	On the other hand, by Theorem \ref{thmreg} since $\nabla \psi(0,r) = (0,0)$ we conclude
	\[
		|\nabla \psi| \leq C \rho^{\tfrac12} = C e^{-\tfrac12 t}.
	\]
	Combining this with \eqref{hatgradpsi} we obtain 
	\[
		|\nabla \hat{\psi}| \leq C e^{-\tfrac32 t} \ \ \text{in} \ \ \hat{S}.
	\]
	Now let $(t_1,\theta_1), (t_2,\theta_2) \in \hat{S}$ be given such that $t \leq t_1<t_2 \leq t+1$. Then from \eqref{hatgradpsi} we find
	\[
		\begin{split}	
			|\nabla \hat{\psi}(t_1,\theta_1) - \nabla \hat{\psi}(t_2,\theta_2)| & \leq C e^{-\tfrac32}(|t_1 - t_2|+|\theta_1-\theta_2|) + C e^{-t}|\nabla \psi(x_1,y_1) - \nabla \psi(x_2,y_2)| \\
			& \leq C e^{-\tfrac32}(|t_1 - t_2|+|\theta_1-\theta_2|) + C e^{-\tfrac32}(|t_1 - t_2|+|\theta_1-\theta_2|)^{\tfrac12} \\
			& \leq C e^{-\tfrac32}(|t_1 - t_2|+|\theta_1-\theta_2|)^{\tfrac12}.
		\end{split}
	\]
	A similar argument is valid for the second-order derivatives and the claim follows from Lemma \ref{lemmaboundary} and \eqref{c2gammaball}.  This finishes the proof.
\end{proof}

\subsection{Flattening of the domain}

The flattening transformation
\[
q = t, \ \ z = \pitwo \frac{\theta + \pitwo}{\zeta(t)+\pitwo}
\]
maps $\hat{S}$ onto the half-strip $S = \{(q,z) \in \R^2: \ \ q \geq -\ln\delta_1, \ \ 0 < z < \pitwo \}$. Thus, the corresponding stream function in new variables is
\[
\bar{\psi}(q,z) = \hat{\psi}(t,\theta),
\]
which solves
\begin{align}
& \left[ \bar{\psi}_q - \frac{z \zeta_q}{\zeta+\pitwo} \bar{\psi}_z \right]_{q} - \frac{z \zeta_q}{\zeta+\pitwo} \left[ \bar{\psi}_q - \frac{z \zeta_q}{\zeta+\pitwo} \bar{\psi}_z \right]_z + \left( \frac{\pitwo}{\zeta+\pitwo} \right)^2 \bar{\psi}_{zz} - e^{-2q}\hat{\omega}(\bar{\psi}) = 0 \ \ \text{in} \ S, \label{sys:psibar:lapp} \\
& \ \bar{\psi} = 0 \ \ \text{on} \ \ z=\pitwo; \ \ \bar{\psi}_z = 0 \ \ \text{on} \ \ z=0. \label{sys:psibar:kin}
\end{align}
The remaining nonlinear boundary relation \eqref{bern1} becomes
\begin{equation} \label{sys:psibar:bern}
\left[ \bar{\psi}_q - \frac{z \zeta_q}{\zeta+\pitwo} \bar{\psi}_z \right]^2 + \left( \frac{\pitwo}{\zeta+\pitwo} \right)^2 \bar{\psi}_z^2 = - 2 e^{-3q} \sin(\zeta) \ \ \text{on} \ \ z=\pitwo.
\end{equation}
In order to reformulate this problem as a first-order system, we introduce an axillary function
\begin{equation} \label{hatPsi}
\hat{\Psi} = \frac{\zeta+\pitwo}{\pitwo} \left[ \bar{\psi}_q - \frac{z\zeta_q}{\zeta+\pitwo} \bar{\psi}_z \right].
\end{equation}
Thus, we can rewrite \eqref{sys:psibar:lapp} as
\begin{align}
& \bar{\psi}_q = \frac{\pitwo}{\zeta+\pitwo}\hat{\Psi} + \frac{z}{\zeta+\pitwo} \zeta_q \bar{\psi}_z,  \label{sys:barpsi:barpsiq}\\
& \hat{\Psi}_q =  \frac {1}{\zeta+\pitwo}\zeta_q (z\hat{\Psi})_z - \frac{\pitwo}{\zeta+\pitwo} \bar{\psi}_{zz} + e^{-2q}\frac{\zeta+\pitwo}{\pitwo} \hat{\omega}(\bar{\psi}) \label{sys:barpsi:hatpsiq}.
\end{align}
Note that we consider $\zeta$ and $\zeta_q$ as coefficients, so that the presence of $q$-derivative of $\zeta$ is not a problem. 

The boundary condition \eqref{sys:psibar:bern} becomes
\begin{equation} \label{hatPsi:bern}
\hat{\Psi}^2 + \bar{\psi}_z^2 = -2\left( \frac{\zeta+\pitwo}{\pitwo} \right)^2 e^{-3q} \sin(\zeta) \ \ \text{on} \ \ z = \pitwo
\end{equation}
and \eqref{sys:psibar:kin} turns into
\begin{equation} \label{hatPsi:kin}
\bar{\psi}_z = 0 \ \ \text{on} \ \ z=0, \ \ \bar{\psi} = 0 \ \ \text{on} \ \ z=\pitwo.
\end{equation}
In the next section we will linearize equations \eqref{sys:barpsi:barpsiq}-\eqref{hatPsi:kin} near an appropriate solution of the model problem.


\subsection{Linearization near the Stokes corner flow}

In order to determine the behaviour of $\bar{\psi}$ near the positive infinity we need to examine the corresponding problem, which is obtained from \eqref{sys:barpsi:barpsiq}-\eqref{hatPsi:kin} by setting $\zeta = -\tfrac{\pi}{6}$; one can also think about passing to the limit $q \to +\infty$ for the coefficients in \eqref{sys:barpsi:barpsiq}-\eqref{hatPsi:kin}. Thus, taking the limit in \eqref{hatPsi} one obtains the relation $\hat{\Psi} = \tfrac23 \bar{\psi}_q$. Therefore, the limiting problem for $\bar{\psi}=\bar{U}$ and $\hat{\Psi} = \tfrac{2}{3} \bar{U}_q$ is the following
\begin{align*}
	& \tfrac{2}{3} \bar{U}_{qq} = - \tfrac32 \bar{U}_{zz}, \ \ && \text{for} \ \ (q,z) \in \R \times (0,\pitwo), \\
	& \tfrac49 \bar{U}_q^2(q,\pitwo) + \bar{U}_z^2(q,\pitwo) = \tfrac49 e^{-3q}, \ \ &&  \text{for} \ \ q \in \R, \\
	& \bar{U}_z(q,0) = \bar{U}(q,\pitwo) = 0, \ \ &&  \text{for} \ \ q \in \R.
	\end{align*}
Separating variables and solving the equations, we find two decaying solutions
\[
\bar{U}_+(q,z) = \tfrac23 e^{-\tfrac32 q} \cos{z}, \ \ \bar{U}_-(q,z) = -\tfrac23 e^{-\tfrac32 q} \cos{z}.
\]
These are the candidates for the leading term in the asymptotics for $\bar{\psi}$. Note that $\bar{\psi}=\bar{U}_+$ and $\zeta = -\tfrac{\pi}6$ determine a solution known as the Stokes corner flow.

\begin{lemma} \label{psiz:decay} Let $\bar{U}_+$ be defined as above. Then 
	\[
	\bar{\psi}_z =  (\bar{U}_+)_z  + o(1)e^{-\thtw q} \ \ \text{on} \ \ z=\pitwo,
	\]
	where $o(1) \to 0$ as $q \to +\infty$. Furthermore, we have $\hat{\Psi}(q,\pitwo) = o(1)e^{-\tfrac32 q}$, where $o(1) \to 0$ as $q \to +\infty$.
\end{lemma}
\begin{proof}
	The latter is a direct consequence of \eqref{hatPsi:bern}. Note that by the definition \eqref{hatPsi} of $\hat{\Psi}$ we have
	\begin{equation}\label{Psiboundary}
	\hat{\Psi} = - \zeta_q \bar{\psi}_z \ \ \text{on} \ \ z = \pitwo.
	\end{equation}
	Using this relation in the Bernoulli equation \eqref{hatPsi:bern} we obtain
	\[
	\begin{split}
		(1+\zeta_q^2)  \bar{\psi}_z^2  & =  -2 \left(\frac{\tfrac{\pi}{3}+(\zeta+\tfrac{\pi}{6})}{\pitwo}\right)^2 e^{-3q}\sin(\zeta+\tfrac{\pi}{6}-\tfrac{\pi}{6})	\\
		&= - \tfrac89 (1+\tfrac{2}{\pi} (\zeta + \tfrac{\pi}{6}))^2 e^{-3q} \left( \tfrac{\sqrt{3}}{2}\sin(\zeta + \tfrac{\pi}{6}) - \tfrac12 + \tfrac12(1-\cos(\zeta + \tfrac{\pi}{6}))  \right).
	\end{split}	
	\]
	Now using the fact that $\zeta + \tfrac{\pi}{6},\zeta_q \to 0$ as $q \to +\infty$, we conclude that
	\begin{equation}\label{barpsiz}
		\bar{\psi}_z^2 = \tfrac49 e^{-3q}(1+ o(1)) \ \ \text{for} \ \ z = \pitwo.
	\end{equation}
	Note that $\bar{\psi}_z(q,\pitwo)$ is negative for large $q$, which follows from the first formula in \eqref{hatgradpsi} (since $\psi_x = -\eta_x \psi_y$ at the boundary). Taking that into account we obtain the desired asymptotics for $\bar{\phi}_z$ by taking the square root in \eqref{barpsiz}.
\end{proof}
Lemma \ref{psiz:decay} shows that $\bar{\psi}$ is close to $\bar{U}$ at the boundary of the domain. As we will justify in next sections this information is enough to obtain higher order asymptotics for $\bar{\psi}$. \\

Now we can linearize equations for $\bar{\psi}$ and $\hat{\Psi}$ near $\bar{U} = \bar{U}_+$ by setting
\[
\bar{\psi}= \bar{U}+\bar{\Phi}, \ \ \hat{\Psi} = \tfrac23 \bar{U}_q  + \bar{\Psi}, \ \ \zeta = - \tfrac16\pi + \xi.
\]
This leads to the following problem for $\bar{\Phi}, \bar{\Psi}$ and $\xi$, where we extract linear terms:
\begin{align}
& \bar{\Phi}_q = \tfrac32 \bar{\Psi} + \frac{9}{2 \pi} \bar{U} \xi + \frac{3}{\pi} z \bar{U}_z \xi_q+\bar{N}_1,  \label{sys:barPhi}\\
& \bar{\Psi}_q =  -\tfrac32 \bar{\Phi}_{zz}-\frac{9}{2\pi} \bar{U} \xi - \frac{3}{\pi}(z \bar{U})_z \xi_q + \bar{N}_2 \label{sys:barPsi}.
\end{align}
Here
\[
	\begin{split}
		\bar{N}_1 & =  - \tfrac32 \frac{\xi}{\xi+\tfrac{\pi}{3}} \hat{\Psi}+ \tfrac3{\pi} \frac{\xi^2}{\xi+\tfrac{\pi}{3}} \bar{U}_q + \frac{z \xi_q}{\xi+\tfrac{\pi}{3}} \bar{\Phi}_z-\tfrac{3}{\pi} \frac{z \bar{U}_z\xi_q}{\xi+\tfrac{\pi}{3}} \xi, \\
		\bar{N}_2 & = \frac{\xi_q}{\xi+\tfrac{\pi}{3}} (z \bar{\Psi})_z + \tfrac{3}{\pi} \frac{(z \bar{U})_z \xi_q}{\xi+\tfrac{\pi}{3}} \xi + \tfrac32 \frac{\xi}{\xi+\tfrac{\pi}{3}} \bar{\Phi}_{zz} + \tfrac9{2\pi} \frac{\bar{U} \xi}{\xi+\tfrac{\pi}{3}}\xi + \tfrac2{\pi}e^{-2q} (\xi+\tfrac{\pi}{3}) \hat{\omega}(\bar{\psi}).
	\end{split}
\]
The Bernoulli equation \eqref{hatPsi:bern} becomes
\[
2 \bar{U}_z \bar{\Phi}_z - \tfrac{8}{3\pi} \xi e^{-3q} + \tfrac{4}{3\sqrt{3}} \xi e^{-3q} = \bar{N}_3
\]
with
\[
\begin{split}
\bar{N}_3=& -\bar{\Psi}^2-\bar{\Phi}_z^2-2\left( \frac{\zeta+\pitwo}{\pitwo} \right)^2 e^{-3q} \left\{ \tfrac{\sqrt{3}}{2}(\sin\xi-\xi) +\tfrac12(1-\cos\xi)\right\} \\
&-\tfrac{8}{\pi^2} \left\{ \tfrac{\sqrt{3}}{2}\xi-\tfrac12+\tfrac{\pi}{\sqrt{3}} \right\} \xi^2 e^{-3q}, 
\end{split}
\]
while 
\[
	\bar{\Phi} = 0 \ \ \text{on} \ \ z=\pitwo, \ \ \bar{\Phi}_z = 0 \ \ \text{on} \ \ z=0.
\]
The linear part of this system can be significantly simplified by introducing new functions $\Phi$ and $\Psi$ from the relations
\begin{equation}\label{defPsiPhi}
\bar{\Phi} = \Phi + \frac{3}{\pi} z \bar{U}_z \xi, \ \ \bar{\Psi} = \Psi -\frac{3}{\pi} (z \bar{U})_z \xi.
\end{equation}
The latter transformation allows to eliminate linear terms with $\xi$ and $\xi_q$ that are not included in $\bar{N}_1,\bar{N}_2$ and $\bar{N}_3$. Thus, plugging \eqref{defPsiPhi} into \eqref{sys:barPhi}-\eqref{sys:barPsi}, we obtain
\begin{align}
& \Phi_q = \tfrac32 \Psi + N_1,  \label{sys:Phi}\\
& \Psi_q =  -\tfrac32 \Phi_{zz} + N_2 \label{sys:Psi},
\end{align}
where
\[
\begin{split}
N_1 & = -\frac{9}{2(\pi+3\xi)} \Psi \xi  + \frac{3z}{\pi(\pi  + 3 \xi)} (\tfrac92 \bar{U}_z \xi^2 - 3 z \bar{U} \xi \xi_q) + \frac{3z}{(\pi + 3 \xi)} \Phi_z \xi_q, \\
N_2 & = \frac{2 \hat{\omega}(\bar{\psi}) (\pi + 3 \xi)}{3 \pi}   e^{-2 q} -\frac{27}{2\pi (\pi + 3\xi)} (\barU+z \bar{U}_z) \xi^2 + \frac{3 \xi_q (z\Psi)_z}{(\pi + 3\xi)} - \frac{9 z (z \barU )_{zz}}{\pi (\pi + 3\xi)} \xi \xi_q + \frac{9 \Phi_{zz} \xi}{2(\pi + 3\xi)}.
\end{split}
\]
The Bernoulli equation takes the following form
\begin{equation} \label{berntmp}
\begin{split}
2 \bar{U}_z \Phi_z + \frac{4 \sqrt{3}}{9} e^{-3q} \xi & = \frac{4 \sqrt{3}}{9\pi^2} e^{-3q} (3\xi+\pi)^2(\sin\xi - \xi) \\
& + \frac{4}{9\pi^2} e^{-3q}(9 \xi^2 \cos \xi + (\cos\xi - 1)(\pi^2+6 \pi \xi)) \\
& - (\Psi - \tfrac32 \bar{U}_z \xi)^2 - (\tfrac{3}{\pi} \bar{U}_z \xi + \Phi_z)^2,
\end{split}
\end{equation}
while the boundary condition at the bottom is unchanged:
\[
\bar{\Phi}_z = 0 \ \ \text{on} \ \ z = 0.
\]
The transformation \eqref{defPsiPhi} has another advantage. Using the homogenous boundary relation $\bar{\Phi} = 0$ at the surface one can recover $\xi$ in terms of the new function $\Phi$ as 
\begin{equation} \label{xiphi}
\xi = - \twth (\barU_z)^{-1} \Phi = e^{\tfrac32 q} \Phi\ \ \text{on} \ \ z = \pitwo.
\end{equation}
This allows to express
\begin{equation}\label{N12linear}
\begin{split}
N_1  = & -\frac{9\xi}{2(\pi+3\xi)} \Psi   - \frac{9 z \sin z \xi}{\pi(\pi  + 3 \xi)} \Phi|_{z=\pitwo} - \frac{6 z^2 \cos z \xi_q}{\pi(\pi  + 3 \xi)} \Phi|_{z=\pitwo} + \frac{3z \xi_q}{(\pi + 3 \xi)} \Phi_z , \\
N_2  = & -\frac{9 \xi (\cos z - z \sin z)}{\pi (\pi + 3\xi)} \Phi|_{z=\pitwo} + \frac{3 \xi_q}{(\pi + 3\xi)} (z\Psi)_z + \frac{6z (\cos z + z \sin z) \xi_q}{\pi (\pi + 3\xi)} \Phi|_{z=\pitwo} \\
& + \frac{9\xi}{2(\pi + 3\xi)}\Phi_{zz} + \frac{2\hat{\omega}(\bar{\psi})e^{-\tfrac12 q}}{\pi}\Phi|_{z=\pitwo} +  \tfrac{2}{3} \omega(1) e^{-2q} + \frac{2 (\hat{\omega}(\bar{\psi})-\hat{\omega}(0)) (\pi + 3 \xi)}{3 \pi}   e^{-2 q}.
\end{split}
\end{equation}
Thus, except the last two terms in $N_2$, we can think $N_1$ and $N_2$ as linear operators of $\Phi$ and $\Psi$. Below we will formalize this idea.

Similarly, we write \eqref{berntmp} as
\begin{equation} \label{bernfinal}
\Phi_z - \tfrac{1}{\sqrt{3}} \Phi = N_3,
\end{equation}
where
\begin{equation}\label{N3lin}
\begin{split}
N_3 = &  \left( \left( \tfrac34 + \tfrac{3}{\pi} \right) e^{-\tfrac32 q} - \tfrac{1}{\sqrt{3} \pi^2} \frac{(3\xi+\pi)^2(\sin\xi - \xi)}{\xi}  -\frac{9 \xi^2 \cos \xi + (\cos\xi - 1)(\pi^2+6 \pi \xi)}{3\pi^2 \xi} \right)\Phi \\
& +  \left( \tfrac34 \Psi e^{\tfrac32 q} + \tfrac32 \xi  \right) \Psi + \left( \tfrac34 \Phi_z e^{\tfrac32 q}-\tfrac3{\pi} \xi \right) \Phi_z.
\end{split}
\end{equation}
As before it can be seen as a linear operator of $\Phi$ and $\Psi$ with decaying coefficients. \\

In what follows we will examine the decay rate of $\Phi$, $\Psi$ and $\xi$. Note that Lemma \ref{lemmahatpsi} only shows that $\Phi$, $\Psi$ are $O(e^{-(3/2) q})$, which is insufficient since this is the order of the leading term $\bar{U}$. Furthermore, we have no information about the decay rate of $\xi$ and $\xi_q$. Such scant information makes the proof quite technical, though the idea is very clear. Let us outline the main steps of the proof. \\

\textbf{Weak decay of $\Phi$ and $\Psi$}. This is the main idea of the proof. Roughly speaking the decay of $\Phi$ and $\Psi$ is determined by the forcing term with $\omega(1)$ in the definition of $N_2$ in \eqref{N12linear}. To formalize that we will define appropriate weighted spaces (with an exponential function as the weight) and will study the model linear operator
	\[
	{\mathcal L}(\Phi,\Psi) = \left(\Phi_q - \tfrac32 \Psi, \Psi_q + \tfrac32 \Phi_{zz}\right).
	\]
	Thus, using the invertibility of ${\mathcal L}$ we will obtain that 
	\[
		\int_{q<q'<q+1}\int_0^{\pitwo} \Phi^2(q',z)+\Phi^2_z(q',z)+\Psi^2(q',z) dz dq' = O(e^{-2\beta q}) 
	\]
	as $ q \to +\infty$ for any $\beta \in (\tfrac32,2)$. In particular this will give an exponential decay for $\xi$. An exponential decay for $\xi_q$ and $\xi_{qq}$ requires additional arguments. \\
	
\textbf{An exponential decay of H\"older norms $\|\Phi\|_{C^{1,\gamma}([q,q+1] \times [0,\pi/2])}$}. For this purpose we will apply Shauder estimates for $\Phi$, solving the second-order elliptic boundary problem. The difficulty here is that we can not exclude $\xi_q$ in terms of $\Phi_q$ by using \eqref{xieta} since this would give $e^{(3/2) q}$ as a coefficient. Instead of that we will show that $\|\xi_q\|_{C^\gamma([q,q+1])} \to 0$ as $q\to+\infty$, which is enough to apply Shauder estimates (with the right-hand side in a divergence form). Next we will obtain a similar statement for the norms $\|\Phi\|_{C^{2,\gamma}([q,q+1] \times [0,\pi/2])}$. This will give the decay for $\xi_q$ and $\xi_{qq}$. \\

\textbf{Higher order asymptotics.} Once we have the decay for $\Phi$ and $\xi$ and their derivatives we can obtain higher order asymptotics. The leading order term for $\Phi$ is determined by the expression with $\omega(1)$ in \eqref{N12linear} and can be found explicitly, which gives
	\[
	\Phi(q,z) = \tfrac{1}{12}(3-2\cos{(\tfrac43 z)}) \hat{\omega}(0) e^{-2q}  + \Phi_{err}(q,z).
	\]
Thus, we need to establish the decay properties for $\Phi_{err}$ and this can be done the same way as before. Though an exponential decay for the function $\Phi_{err}$ by itself can be obtained by using the invertibility of ${\mathcal L}$, the decay for higher order derivatives will require additional arguments.

\subsection{A weak decay for $\Phi$ and $\Phi$}

It follows from \eqref{psiz:decay} that $\Phi(q,\pitwo) e^{ (3/2)q}$ decays to zero as $q \to +\infty$. Below we will prove the following statement:

\begin{proposition}\label{holderdecay} For any $\gamma \in (0,\tfrac12)$ we have that
	\[
		\|\Phi\|_{C^{1,\gamma}([q,q+1]\times[0,\pitwo])}, \ \|\Psi\|_{C^{\gamma}([q,q+1]\times[0,\pitwo])}  = e^{ -\tfrac32 q} o(1),
	\]
	where $o(1) \to 0$ as $q \to +\infty$.
\end{proposition}
This improves the statement of Lemma \ref{psiz:decay}. Before proving the claim we need to obtain a similar decay property for $\xi$:

\begin{lemma} \label{lemmaxi} For any $\gamma \in (0,\tfrac12)$ the norms $\|\xi\|_{C^{1,1/2}([q,q+1])}$ are uniformly bounded and
	\begin{equation}\label{xiholder}
		\|\xi\|_{C^{1,\gamma}([q,q+1])} \to 0
	\end{equation}
	as $q \to +\infty$.
\end{lemma}
\begin{proof}
	To prove the claim it is enough to show that the norm $\|\xi_q\|_{C^{1/2}([q,q+1])}$ is bounded by a constant independent of $q$. Indeed, assuming that is true we can use the inequality
	\[
	\begin{split}
	\frac{|\xi_q(q_1)-\xi_q(q_2)|}{|q_1-q_2|^\gamma} & = \frac{|\xi_q(q_1)-\xi_q(q_2)|}{|q_1-q_2|^{\tfrac12}} |q_1-q_2|^{\tfrac12-\gamma} \\
	& \leq \begin{cases}
	& \|\xi_q\|_{C^{\tfrac12}([q,q+1])}\|\xi_q\|_{L^{\infty}[q,q+1]}^{\tfrac12-\gamma}, \ \ |q_1-q_2| \leq |\xi_q(q_1)-\xi_q(q_2)| \\
	& 2 \|\xi_q\|_{L^{\infty}[q,q+1]}^{1-\gamma}, \ \  |q_1-q_2| > |\xi_q(q_1)-\xi_q(q_2)|
	\end{cases}
	\end{split}
	\]
	Thus, if $\|\xi_q\|_{C^{1/2}([q,q+1])}$ are uniformly bounded, then the right-hand side from the above estimate tends to zero since we already known that $\xi_q \to 0$ as $q \to +\infty$. Note that $\xi_q(q) = \zeta_t(q)$ so we need to estimate the quantity $|\zeta_t(t_1) - \zeta_t(t_2)|$ for arbitrary $t_1,t_2 \in [t,t+1]$. For that purpose we use the formula \eqref{xieta} to obtain
	\begin{equation}\label{zetaest1}
	\begin{split}
	|\zeta_t(t_1) - \zeta_t(t_2)| & \leq C ( |\nabla \psi (x_1,\eta(x_1))-\nabla \psi (x_2,\eta(x_2))| \\
	& + C \max_{x_1 \leq x \leq x_2} |\nabla \psi(x,\eta(x))| |\zeta(t_1)-\zeta(t_2)| ) e^{\tfrac12 t}.
	\end{split}
	\end{equation}
	Here we used that the quantity
	\[
	-\psi_x \sin{\zeta} + \psi_y \cos{\zeta} = \left(\tfrac{1}{2\sqrt{3}}+\tfrac{\sqrt{3}}{2} + o(1)\right) e^{-\tfrac12 t},
	\]
	where $o(1) \to 0$ as $t \to +\infty$. Note that $|\nabla \psi(x,\eta(x))| = O(e^{-\tfrac12 t})$ by  Theorem \ref{thmreg} and
	\begin{equation}\label{zetaest2}
	|\nabla \psi (x_1,\eta(x_1))-\nabla \psi (x_2,\eta(x_2))| \leq C e^{-\tfrac12 t} |t_1-t_2|^{\tfrac12}.
	\end{equation}
	Finally, since 
	\[
	|\zeta(t_1)-\zeta(t_2)| \leq C |t_1-t_2|
	\]
	we obtain from \eqref{zetaest1} and \eqref{zetaest2} that
	\[
	|\zeta_t(t_1) - \zeta_t(t_2)| \leq C |t_1 - t_2|^{\tfrac12}.
	\]
	This shows that \eqref{xiholder} is true for $\xi_q$. The remaining estimate for $\xi$ is now trivial.
\end{proof}

\begin{proof}[Proof of Proposition \ref{holderdecay}]
	Let $\gamma \in (0,\tfrac12)$ be given and assume that the claim is false so that there exists a sequence $\{q_j\}_{j=1}^{+\infty}$ accumulating at $+\infty$ and such that
	\begin{equation}\label{Phieps}	
		\|\Phi\|_{C^{1,\gamma}([q_j,q_j+1]\times[0,\pitwo])} \geq e^{-\tfrac32 q_j} \epsilon
	\end{equation}
	for some $\epsilon > 0$ and all $j \geq 1$. Then we consider functions
	\[
		\Phi^{(j)}(q,z) = \Phi(q+q_j,z)e^{\tfrac32(q+q_j)}, \ \ \Psi^{(j)}(q,z) = \Psi(q+q_j,z)e^{\tfrac32(q+q_j)}, \ \ \xi^{(j)}(q) = \xi(q+q_j).
	\]
	Note that by the definition we have
	\[
		\Phi = \bar{\psi} - \bar{U}- \tfrac3{\pi} z \bar{U}_z \xi
	\]
	so that Lemma \ref{lemmahatpsi} and Lemma \ref{lemmaxi} give
	\begin{equation}\label{hol12}	
		\|\Phi^{(j)}\|_{C^{1,\tfrac12}([0,+\infty) \times [0,\pitwo])},  \ \|\Psi^{(j)}\|_{C^{\tfrac12}([0,+\infty) \times [0,\pitwo])},  \ \|\xi^{(j)}\|_{C^{1,\tfrac12}([0,+\infty))} < + \infty.
	\end{equation}
	Then by the compactness we can find a subsequence $\{j_k\}$ such that functions $\Phi^{(j_k)}$, $\Psi^{(j_k)}$ and $\xi^{(j_k)}$ converge in every space $C^{1,\gamma}(I \times [0,\pitwo])$, $C^{\gamma}(I \times [0,\pitwo])$ and $C^{1,\gamma}(I)$ respectively for all compact intervals $I \subset \R \times [0,\pitwo]$ (though some finite number of functions may not be defined). We denote the limiting functions by $\Phi^{\infty}$, $\Psi^{\infty}$ and $\xi^{\infty}$. Note that $\xi^{\infty}$ is identically zero. These limiting functions are subject to certain equations that we will derive below. For this purpose we need to exploit a weak form of equations \eqref{sys:Phi}-\eqref{sys:Psi}.
	
	Let us multiply equations \eqref{sys:Phi}-\eqref{sys:Psi} by $e^{3/2 q}$ first, and then by some test functions $\varPhi(q+q_j,z)$ and $\varPsi_j(q+q_j,z)$ with compact supports. This leads to the equations
	\[
		\begin{split}
		-\iint \Phi^{(j)}(q,z) \varPhi_q(q,z) dqdz & = \tfrac32 \iint \left\{ \Phi^{(j)}(q,z)+ \Psi^{(j)}(q,z) \right\} \varPhi(q,z) dqdz \\
		& + \iint N_1(q,z) e^{\tfrac32 q} \varPhi(q+q_j,z) dqdz, \\
		-\iint \Psi^{(j)}(q,z) \varPsi_q(q,z) dqdz & = \tfrac32 \iint \Psi^{(j)}(q,z)dqdz +\tfrac32 \iint \Phi^{(j)}_{z}(q,z) \varPsi_z(q,z) dqdz \\
		& + \iint N_2(q,z) e^{\tfrac32 q} \varPsi(q+q_j,z) dqdz.		
		\end{split}
	\]
	The definitions of $N_1$ and $N_2$ and \eqref{hol12} imply that the integral terms with $N_1$ and $N_2$ tend to zero as $j \to +\infty$. Thus, the limiting functions solve
	\[
		\begin{split}
			-\iint \Phi^{\infty}(q,z) \varPhi_q(q,z) dqdz & = \tfrac32 \iint \left\{ \Phi^{\infty}(q,z)+ \Psi^{\infty}(q,z) \right\} \varPhi(q,z) dqdz \\
			-\iint \Psi^{\infty}(q,z) \varPsi_q(q,z) dqdz & = \tfrac32 \iint \Psi^{\infty}(q,z)dqdz +\tfrac32 \iint \Phi^{\infty}_{z}(q,z) \varPsi_z(q,z) dqdz.		
		\end{split}		
	\]
	This shows that $\Phi^{\infty}$ and $\Psi^{\infty}$ are infinitely smooth in $\R \times (0,\pi/2)$ and $\Phi^{\infty}$ is subject to
	\[
		\Phi^{\infty}_{qq} - 3 \Phi^{\infty}_q + \tfrac94 \Phi^{\infty} + \tfrac94 \Phi^{\infty}_{zz} = 0.
	\]
	On the other hand $\Phi_z^{\infty} = 0$ on $z=0$, while $\Phi_z^{\infty}=\Phi^{\infty} = 0$ on $z=\pitwo$, which follows from \eqref{bernfinal} and Lemma \ref{psiz:decay}. Together this guarantees that $\Phi^{\infty}$ is zero identically, though \eqref{Phieps} shows the opposite, leading to a contradiction. In a similar way one proves the claim about $\Psi$.
\end{proof}

\subsection{Reduction to homogenous boundary conditions} 

For a further analysis of \eqref{sys:Phi}-\eqref{sys:Psi} we need to replace \eqref{bernfinal} by a homogenous relation (with $N_3 = 0$). For this purpose we introduce
\begin{equation} \label{bernhomtmp}
\Phi^{\star}(q,z) = \Phi(q,z) + \tfrac{2}{\pi}\int_z^{\pitwo} \tau  N_3^\star(q)(\Phi(q,\cdot),\Psi(q,\cdot))(\tau) d \tau, \ \ \Psi^\star(q,z) = \Psi(q,z),
\end{equation}
where
\[
	\begin{split}
	N_3^\star(q)(\varPhi,\varPsi) = &  \left( \left( \tfrac34 + \tfrac{3}{\pi} \right) e^{-\tfrac32 q} - \tfrac{1}{\sqrt{3} \pi^2} \frac{(3\xi+\pi)^2(\sin\xi - \xi)}{\xi}  -\frac{9 \xi^2 \cos \xi + (\cos\xi - 1)(\pi^2+6 \pi \xi)}{3\pi^2 \xi} \right) \varPhi \\
	& +  \left( \tfrac34 \varPsi e^{\tfrac32 q} + \tfrac32 \xi \right) \varPsi + \left( \tfrac34 \varPhi_z e^{-\tfrac32 q}-\tfrac3{\pi} \xi \right) \varPhi_z.
	\end{split}
\]
The main purpose of \eqref{bernhomtmp} is that the nonlinear boundary relation \eqref{bernfinal} is transformed to
\begin{equation} \label{bernhom}
\begin{split}
\Phi^{\star}_z - \tfrac{1}{\sqrt{3}} \Phi^{\star} &= 0 \ \ \text{on} \ \ z = \pitwo,\\
\Phi^{\star}_z &= 0 \ \ \text{on} \ \ z = 0.
\end{split}
\end{equation}
Note that $N_3^\star(\Phi,\Psi) = N_3(\Phi,\Psi)$ for $z = \pitwo$. At the same time $N_3^\star$ is defined as a nonlinear operator from $H^{2}(0,\pitwo) \times H^{1}(0,\pitwo)$ to $H^{1}(0,\pitwo)$. Thus, for all sufficiently large $q$ our definition \eqref{bernhomtmp} determines a near-identical transformation in $H^{2}(0,\pitwo) \times H^{1}(0,\pitwo)$, which follows from Proposition \ref{holderdecay}. This allows to express
\begin{equation}\label{PhitoPhistar}
	(\Phi(q,z),\Psi(q,z)) = (B^{(1)}(q)(\Phi^\star(q,\cdot), \Psi^\star(q,\cdot))(z), B^{(2)}(q)(\Phi^\star(q,\cdot), \Psi^\star(q,\cdot))(z)),
\end{equation}
where 
\[
	\begin{split}
	& B^{(1)}(q):H^{2}(0,\pitwo) \times H^{1}(0,\pitwo) \to H^{2}(0,\pitwo), \\
	& B^{(2)}(q):H^{2}(0,\pitwo) \times H^{1}(0,\pitwo) \to H^{1}(0,\pitwo)
	\end{split}
\]
are bounded nonlinear operators and their norms are uniformly bounded for large $q$.

Let us differentiate \eqref{bernhomtmp} with respect to $q$-variable, which leads to
\begin{equation}\label{Phistarq}
\Phi^{\star}_q = \Phi_q + \tfrac{2}{\pi} \int_z^{\pitwo} \tau \frac{\partial N_3^\star}{\partial q} (\Phi,\Psi)  + N_4^\star d\tau,
\end{equation}
where 
\[
	\begin{split}
		N_4^\star = &  \left( \left( \tfrac34 + \tfrac{3}{\pi} \right) e^{-\tfrac32 q} - \tfrac{1}{\sqrt{3} \pi^2} \frac{(3\xi+\pi)^2(\sin\xi - \xi)}{\xi}  -\frac{9 \xi^2 \cos \xi + (\cos\xi - 1)(\pi^2+6 \pi \xi)}{3\pi^2 \xi} \right) \Phi_q \\
		& +  \left( \tfrac32 \Psi e^{-\tfrac32 q} + \tfrac32 \xi e^{-\tfrac32 q} \right) \Psi_q + \left( \tfrac32 \Phi_z e^{-\tfrac32 q}-\tfrac4{\pi}e^{-\tfrac32 q} \right) \Phi_{zq}.
	\end{split}
\]
Now we can replace $\Phi_q$ and $\Psi_q$ in the latter formula by using \eqref{sys:Phi} and \eqref{sys:Psi}. Furthermore, in the corresponding expressions for $N_1$ and $N_2$ given by \eqref{N12linear} we replace $\Phi$ and $\Psi$ by using \eqref{PhitoPhistar}. In the same way we replace all remaining occurrences of $\Phi$ and $\Psi$ in \eqref{Phistarq}. After the same procedure for $\Psi^\star_q$ we obtain that
\begin{align}
& \Phi^{\star}_q = \tfrac32 \Psi^{\star} + N_1^{\star}(q)(\Phi^{\star}(q,\cdot),\Psi^{\star}(q,\cdot)),  \label{sys:Phistar}\\
& \Psi_q^\star =  -\tfrac32 \Phi_{zz}^\star + N_2^\star(q)(\Phi^{\star}(q,\cdot),\Psi^{\star}(q,\cdot)) + f, \label{sys:Psistar}
\end{align}
where $N_1^\star(q):H^{2}(0,\pitwo) \times H^{1}(0,\pitwo) \to H^{1}(0,\pitwo)$ and $N_2^\star(q):H^{2}(0,\pitwo) \times H^{1}(0,\pitwo) \to L^2(0,\pitwo)$ are bounded nonlinear operators for all sufficiently large $q$ and the corresponding norms satisfy
\begin{equation}\label{decayN12}
	\|N_1^\star(q)\|, \|N_2^\star(q)\| \to 0 \ \ \text{as} \ \ q \to +\infty.
\end{equation}
More precisely, $N_1^\star$, $N_2^\star$ and $f$ are given by
\begin{equation}\label{N12star}
\begin{split}
N_1^\star  &= -\frac{9\xi}{2(\pi+3\xi)} B^{(2)}   - \frac{9 z \sin z \xi}{\pi  + 3 \xi} (B^{(1)})|_{z=\pitwo} - \frac{6 z^2 \cos z \xi_q}{\pi  + 3 \xi} (B^{(1)})|_{z=\pitwo} + \frac{3z \xi_q}{(\pi + 3 \xi)} (B^{(1)})_z , \\
N_2^\star  &= -\frac{9 \xi (\cos z - z \sin z)}{\pi (\pi + 3\xi)} (B^{(1)})|_{z=\pitwo} + \frac{3 \xi_q}{(\pi + 3\xi)} (z B^{(2)} )_z + \frac{6z (\cos z + z \sin z) \xi_q}{\pi (\pi + 3\xi)} (B^{(1)})|_{z=\pitwo} \\
& + \frac{9\xi}{2(\pi + 3\xi)}(B^{(1)})_{zz} + \frac{2\hat{\omega}(\bar{\psi})e^{-\tfrac12 q}}{\pi}(B^{(1)})|_{z=\pitwo}, \\
 f &=  \tfrac{2}{3} \omega(1) e^{-2q} + \frac{2 (\hat{\omega}(\bar{\psi})-\hat{\omega}(0)) (\pi + 3 \xi)}{3 \pi}   e^{-2 q}.
\end{split}
\end{equation}

Note that the system \eqref{sys:Phistar}-\eqref{sys:Psistar} is equipped with the homogenious boundary relations \eqref{bernhom}.

\subsection{The model linear problem} At the positive infinity the nonlinear system \eqref{sys:Phistar}-\eqref{sys:Psistar} is reduced to the linear problem
\begin{align}
& \varPhi_q^\star = \tfrac32 \varPsi^\star,  \nonumber\\
& \varPsi_q^\star =  -\tfrac32 \varPhi^\star_{zz} \nonumber,
\end{align}
where $\varPhi^\star$ is subject to the boundary relations \eqref{bernhom}. The corresponding spaces for $\varPhi^\star$ and $\varPsi^\star$ are
\[
\begin{split}
& X_1 = \{ \varPhi \in H^{2}(0,\pitwo) :  \ \varPhi_z(0) = 0, \ \ \varPhi_z(\pitwo) - \tfrac{1}{\sqrt{3}}\varPhi(\pitwo) = 0\}, \\
& X_2 = H^1(0,\pitwo). 
\end{split}
\]
Furthermore, we define the range spaces as
\[
Y_1 = H^1(0,\pitwo), \ \ Y_2 = L^2(0,\pitwo).
\]
Next we consider the corresponding linear operator ${\mathcal L} : W^{1,2}_{loc}(\R; X_1 \times X_2) \to L^2_{loc}(\R; Y_1 \times Y_2)$ given by
\[
{\mathcal L}(\varPhi,\varPsi) = \left(\varPhi_q - \tfrac32 \varPsi, \varPsi_q + \tfrac32 \varPhi_{zz}\right).
\]
For a given $\beta > 0$ and an interval $I \subset \R$ we define weighted spaces $W^{1,2}_{\beta}(I; X_1 \times X_2)$ and $L^2_{\beta}(I; Y_1 \times Y_2)$ as subspaces of $W^{1,2}_{loc}(I; X_1 \times X_2)$ and $L^2_{loc}(I; Y_1 \times Y_2)$ of functions with finite norms
\[
\begin{split}
& \|x\|_{W^{1,2}_{\beta}(I)} = \left(\int_I e^{2\beta q} \left[ \|x(q)\|_{X_1 \times X_2}^2 + \|D_q x(q)\|_{Y_1 \times Y_2}^2 \right] dq \right)^{\tfrac12} \\
& \|x\|_{L^2_{\beta}(I)} = \left(\int_I e^{2\beta q} \|x(q)\|_{Y_1 \times Y_2}^2 dq \right)^{\tfrac12}.
\end{split}
\]
Let us study the kernel of ${\mathcal L} : W^{1,2}_{loc}(\R; X_1 \times X_2) \to L^2_{loc}(\R; Y_1 \times Y_2)$. Separating variables one finds that the kernel is spanned by the following functions
\[
\varphi_j(z) e^{\pm \tfrac32 \tau_j q}, \ \ j=0,1,2,...,
\]
where $\varphi_j$ and $\mu_j:=\tau_j^2$ are the corresponding eigenpairs of the following Sturm--Liouville problem
\begin{equation}\label{SLproblem}
-\varphi'' = \mu_j \varphi, \ \ z \in (0,\pitwo); \ \ \varphi_z(0) = 0, \ \ \varphi_z(\pitwo)-\tfrac1{\sqrt{3}} \varphi(\pitwo) = 0.
\end{equation}
The latter has a discrete spectrum accumulating at the positive infinity, while the first eigenvalue $\mu_0$ is negative. The second eigenvalue $\mu_1 = \tau_1^2$ can be found as the smallest positive solution to
\[
\tau_1 = - \tfrac{1}{\sqrt{3}} \cot(\pitwo \tau_1)
\] 
and is approximately given by
\[
\tau_1 \approx 1.8.
\]
Note that the numbers $\tau_j$ are the same as $1+\beta_j$ in Amick and Fraenkel \cite{Amick1987a}, where similar asymptotics were studied in the infinite depth case.\\

Our proof will be based on the following basic fact about ${\mathcal L}$. 

\begin{proposition}\label{p:L} The operator ${\mathcal L}:W^{1,2}_{\beta}(\R; X_1 \times X_2) \to L^2_{\beta}(\R; Y_1 \times Y_2)$ is invertible for any $\beta > 0$, provided $\beta \neq \thtw \tau_j$, $j \geq 1$.
\end{proposition}

The statement follows directly from \cite[Theorem 2.4.1]{Kozlov1999}. 

\subsection{A higher-order exponential decay}

In order to employ Proposition \ref{p:L} we need the following preliminary result.

\begin{proposition} \label{l:inspace}
	There exists $\beta_0 \in (0,3/2)$ such that $(\chi \Phi^\star, \chi \Psi^\star) \in W^{1,2}_{\beta_0}(\R;X_1 \times X_2)$ for some cut-off function $\chi$.
\end{proposition}
By a cut-off function we mean a smooth function $\chi(q)$ such that $\chi(q) = 1$ for $q > q_0+1$ and $\chi(q) = 0$ for $q < q_0$ with some $q_0>0$.
\begin{proof}
	To prove the claim we apply Shauder type estimates to the system  \eqref{sys:Phistar}-\eqref{sys:Psistar}. Thus, for intervals $I = [q,q+1]$ and $I_1 = [q-1,q+1]$ we apply \cite[Lemma 2.9.1]{Kozlov1999} and obtain a local estimate
	\[
	\begin{split}
	  \|(\Phi^\star,\Psi^\star)\|_{W^{1,2}(I; X_1 \times X_2)} \leq  C \Big\{ & \|N_1^\star\|_{L^{2}(I_1;Y_1)} +\|N_2^\star\|_{L^{2}(I_1;Y_2)} + \|f\|_{W^{1,2}(I_1; Y_2)} + \\ 
	 &  \|\Phi^\star\|_{L^{2}(I_1;L^2(0,\pitwo)))}+ \|\Psi^\star\|_{L^{2}(I_1;L^2(0,\pitwo)))} \Big\}.
	\end{split}
	\]
	Let $\chi$ and $\beta_0 \in (0,3/2)$ be given. Then using the latter inequality we can estimate
	\begin{equation}\label{shauder1}
		\begin{split}
		\|(\chi \Phi^\star, \chi \Psi^\star)\|_{W^{1,2}_{\beta_0}}^2 & \leq C \int_{q_0-1}^{+\infty} \|(\Phi^\star,\Psi^\star)\|_{W^{1,2}(I; X_1 \times X_2)}^2 e^{2\beta_0 q} dq \\
		& \leq C \int_{q_0-1}^{+\infty} \max_{I_1} \left\{\|N_1^\star(q)\|^2+\|N_2^\star(q)\|^2\right\} \cdot \|(\Phi^\star,\Psi^\star)\|_{W^{1,2}(I_1; X_1 \times X_2)}^2 \\
		& + C e^{-(\tfrac{3}{2}-\beta_0) q_0}.
		\end{split}		 
	\end{equation}
	Here we used the fact that 
		\[
		\|\Phi^\star\|_{L^\infty(I_1\times[0,\pitwo])}, \|\Psi^\star\|_{L^\infty(I_1\times[0,\pitwo])} = O(e^{-\tfrac32 q}),
		\]
	which is a consequence of Lemma \eqref{lemmahatpsi}. Note that in view of \eqref{decayN12} we find from \eqref{shauder1} that
	\[
			\|(\chi \Phi^\star, \chi \Psi^\star)\|_{W^{1,2}_{\beta_0}}^2 \leq C + C \epsilon \|(\chi \Phi^\star, \chi \Psi^\star)\|_{W^{1,2}_{\beta_0}}^2,
	\]
	where $\epsilon \to 0$ as $q_0 \to +\infty$. Thus, choosing $\epsilon$ to be small enough and subtracting the corresponding term we obtain the desired estimate.
\end{proof}

Now we are ready to establish a higher-order weak decay for $\Phi^\star$ and $\Psi^\star$. More precisely we prove

\begin{proposition}\label{phistardecay} For any $\beta \in (0,2)$ there exists a cut-off function $\chi$ such that $(\chi \Phi^\star, \chi \Psi^\star) \in W^{1,2}_{\beta}(\R;X_1 \times X_2)$.
\end{proposition}
\begin{proof}
	Let us multiply \eqref{sys:Phistar}-\eqref{sys:Psistar} by some cut-off function $\chi$ and write the corresponding equations as
	\[
		{\mathcal L}(\chi \Phi^\star, \chi \Psi^\star) - {\mathcal N}(\chi \Phi^\star, \chi \Psi^\star) = (\chi_1,\chi_2 + \chi f),
	\]
	where $\chi_1$ and $\chi_2$ are cut-off functions and 
	\[
		f =  \tfrac{2}{3} \omega(1) e^{-2q} + \frac{2 (\hat{\omega}(\bar{\psi})-\hat{\omega}(0)) (\pi + 3 \xi)}{3 \pi}   e^{-2 q};
	\]
	the nonlinear operator $N$ is defined as
	\[
		{\mathcal N}(\varPhi, \varPsi) = \big(\chi N_1^\star (\varPhi, \varPsi),\chi N_2^\star (\varPhi, \varPsi)\big).
	\]
	Given $\beta \in (0,2)$ the operator ${\mathcal L}:W^{1,2}_{\beta}(\R; X_1 \times X_2) \to L^2_{\beta}(\R; Y_1 \times Y_2)$ is invertible by Proposition \ref{p:L} since the interval $(0,2)$ is free from eigenvalues $\mu_j = (3/2)\tau_j$. On the other hand, the norm of the operator
	\[
		{\mathcal N}:W^{1,2}_{\beta}(\R; X_1 \times X_2) \to L^2_{\beta}(\R; Y_1 \times Y_2)
	\]
	is small, provided $q_0$ (in the definition of the cut-off function) is large enough. Therefore, the operator ${\mathcal L} - {\mathcal N}$ is invertible in the latter spaces so that the equation
	\[
			{\mathcal L}(\varPhi^\star, \varPsi^\star) - {\mathcal N}(\varPhi^\star, \varPsi^\star) = (\chi_1,\chi_2 + \chi f)
	\]
	has a unique solution in $(\varPhi^\star, \varPsi^\star) \in W^{1,2}_{\beta}(\R; X_1 \times X_2)$. Finally, the unique solubility in $W^{1,2}_{\beta_0}(\R; X_1 \times X_2)$ (with $\beta_0$ from Proposition \ref{l:inspace}) and Proposition \ref{l:inspace} give that $\varPhi^\star = \chi \Phi^\star$ and $\varPsi^\star = \chi \Psi^\star$. This finishes the proof.
\end{proof}

It immediately follows from Proposition \ref{phistardecay} that $(\chi \Phi, \chi \Psi) \in W^{1,2}_{\beta}(\R;H^2(0,\pi/2) \times H^1(0,\pi/2))$ for any $\beta \in (0,2)$. This shows that
\begin{equation}\label{phiexpdecay}
	\Phi(q,z) = O(e^{-\beta q}) \ \ \text{as} \ \ q \to +\infty
\end{equation}
uniformly in $z$. Thus, we find from \eqref{xiphi} that
\begin{equation}\label{xiexpdecay}
	\xi(q) = O(e^{-\beta_0 q})
\end{equation}
for any $\beta_0 \in (0,1/2)$. 

\begin{proposition}\label{C2gamma} For any $\gamma \in (0,1/2)$ and $\beta \in (0,2)$ there exist $C>0$ and $q_0>0$ such that
	\[
		\|\Phi\|_{C^{2,\gamma}([q,q+1]\times [0,\pitwo])} \leq C e^{-\beta q}
	\]
	for all $q \geq q_0$.
\end{proposition}
\begin{proof}
	Let $\gamma \in (0,1/2)$ and $0 < \beta < 2$ be given. First we derive a second-order equation for the function $\Phi$ by differentiating \eqref{sys:Phi} and using \eqref{sys:Psi} to replace $\Psi_q$. This gives
	\[
		\Phi_{qq}+\tfrac92 \Phi = (N_1)_q + \tfrac32 N_2.
	\]
	Note that $N_2$ is given in a divergence from: $N_2 = (N_2^\circ)_z$. For a given $q$ we consider intervals $I=[q+1]$ and $I_1 = [q-1,q+1]$ and apply Theorem 9.3 from \cite{AgmonDouglisNirenberg59} to conclude
	\begin{equation}\label{shaud1}
	\begin{split}
		\|\Phi\|_{C^{2,\gamma}(I \times [0,\pitwo])} \leq C \Big\{ & \|N_1\|_{C^{1,\gamma}(I_1 \times [0,\pitwo])} + \|N_2\|_{C^{\gamma}(I_1 \times [0,\pitwo])} + \\	
		&\|N_3\|_{C^{1,\gamma}(I_1)} + \|\Phi\|_{L^{2}(I_1 \times [0,\pitwo])}  \Big\},
	\end{split}
	\end{equation}
	where the constant is independent of $q$. To simplify the notations we will use the following convention:
		\[
			\| \cdot \|_{C^{k,\gamma}(I_1 \times [0,\pitwo])} = \| \cdot \|_{k, \gamma; I_1}, \ \ k=0,1,2.
		\]
	Let us estimate the right-hand side in \eqref{shaud1}. First, we note that \eqref{sys:Phi} implies
	\[
		\|\Psi\|_{1,\gamma;I_1} \leq C \big\{ \|\Phi\|_{2,\gamma;I_1} + \|\xi\|_{1,\gamma; I_1} \|\Phi\|_{2,\gamma; I_1} + \|\xi_q\|_{1,\gamma; I_1}\|\Phi\|_{1,\gamma; I_1} \big\}.
	\]
	On the other hand, we find from \eqref{xiphi} that
	\[
		\|\xi_q\|_{1,\gamma; I_1} \leq \|\Phi\|_{2,\gamma;I_1} e^{\tfrac32 q}.
	\]
	Note that $\|\Phi\|_{1,\gamma; I_1} e^{\tfrac32 q} \to 0$ as $q \to +\infty$ by Proposition \ref{holderdecay}, while $\|\xi\|_{1,\gamma; I_1} \to 0$ by Lemma \ref{lemmaxi} so that
	\[
		\|\Psi\|_{1,\gamma;I_1} \leq C \|\Phi\|_{2,\gamma;I_1}.
	\]
	In a similar fashion we estimate the right-hand side in \eqref{shaud1} and conclude
	\begin{equation}\label{shaud2}
		\|\Phi\|_{2,\gamma;I} \leq C \Big\{ o(1) \|\Phi\|_{2,\gamma;I_1} + e^{-\beta q} \Big\},
	\end{equation}
	where $o(1)$ is bounded and $o(1) \to 0$ as $q \to +\infty$. Here we used \eqref{phiexpdecay} to estimate the corresponding $L^2$-norm. Taking into account that the norms $\|\Phi\|_{2,\gamma;I}$ are uniformly bounded by \eqref{lemmahatpsi} we conclude the desired estimate from \eqref{shaud2} by the iteration.
\end{proof}

\subsection{Explicit asymptotics for $\Phi$ and $\Psi$}

To find the next leading term for $\Phi^\star$ and $\Psi^\star$ we need to consider an inhomogeneous problem which is obtain from \eqref{sys:Phistar}-\eqref{sys:Psistar} by setting $N_1^\star$ and $N_2^\star$ to zero and $f$ to $\tfrac23 \omega(1)e^{-2q}$. The corresponding equations are 
\begin{align*}
& U^{\star}_q = \tfrac32 V^{\star}, \\
& V^{\star}_q =  -\tfrac32 U^{\star}_{zz} + \tfrac23 \hat{\omega}(0) e^{-2q}.
\end{align*}
Separating variables one finds a solution
\[
U^\star = \tfrac{1}{12}(3-2\cos{(\tfrac43 z)}) \hat{\omega}(0) e^{-2q}.
\]
We will show below that $U^\star$ determines $\Phi^\star$ up to the leading order. For this purpose we put
\[
\Phi^\star = U^\star+\Phi^\dagger , \ \ \Psi^\star = \tfrac23 U^\star + \Psi^\dagger
\]
and obtain from \eqref{sys:Phistar}-\eqref{sys:Psistar} equations for the next order terms $\Phi^\dagger$ and $\Psi^\dagger$:
\begin{align}
& \Phi^{\dagger}_q = \tfrac32 \Psi^{\dagger} + N_1^\star(q)(\Phi^{\dagger},\Psi^{\dagger}) + f_1^\dagger,  \label{sys:Phidag}\\
& \Psi^{\dagger}_q =  -\tfrac32 \Phi^{\dagger}_{zz} + N_2^\star(q)(\Phi^{\dagger},\Psi^{\dagger}) + f_2^\dagger \label{sys:Psidag},
\end{align}
where
\[
	f_1^\dagger = N_1^\star(q)(U^{\star},V^{\star}), \ \ f_2^\dagger = N_2^\star(q)(U^{\star},V^{\star}) + \frac{2 (\hat{\omega}(\bar{\psi})-\hat{\omega}(0)) (\pi + 3 \xi)}{3 \pi}   e^{-2 q}.
\]

\begin{lemma}\label{phidaggerdecay}
	For any $\beta \in (0,5/2)$ there exists a cut-off function $\chi$ such that $(\chi \Phi^\dagger, \chi \Psi^\dagger) \in W^{1,2}_{\beta}(\R;X_1 \times X_2)$.
\end{lemma}
\begin{proof}
	One argues the same way as in the proof of Proposition \ref{phistardecay}. We only note that $(f_1^\dagger,f_2^\dagger) \in W^{1,2}_{\beta}(\R;X_1 \times X_2)$ for any $\beta \in (0,5/2)$, while $5/2 < (3/2)\tau_1$.
\end{proof}

Note that we can not apply Shauder estimates directly for the function $\Phi^\dagger$ in order to establish a decay for the first- and second-order derivatives. Instead of that we express
\begin{equation}\label{asympPhi}
	\Phi = U^\star + \tilde{\Phi}, \ \ \Psi = V^\star + \tilde{\Psi}.
\end{equation}
It follows from \eqref{bernhomtmp} and Proposition \ref{C2gamma} that $(\chi \tilde{\Phi}, \chi \tilde{\Psi}) \in W^{1,2}_{\beta}(\R;X_1 \times X_2)$ for any $\beta \in (0,5/2)$. On the other hand, $\tilde{\Phi}$ and $\tilde{\Psi}$ solve the problem
\begin{align}
& \tilde{\Phi}_q = \tfrac32 \tilde{\Psi} + N_1(q)(\tilde{\Phi},\tilde{\Psi}) + \tilde{f}_1,  \label{sys:Phidag}\\
& \tilde{\Psi}_q =  -\tfrac32 \tilde{\Phi}_{zz} + N_2(q)(\tilde{\Phi},\tilde{\Psi}) + \tilde{f}_2 \label{sys:Psidag},
\end{align}
where
\[
\tilde{f}_1 = N_1(q)(U^{\star},V^{\star}), \ \ \tilde{f}_2 = N_2(q)(U^{\star},V^{\star}) + \frac{2 (\hat{\omega}(\bar{\psi})-\hat{\omega}(0)) (\pi + 3 \xi)}{3 \pi}   e^{-2 q}.
\]
Furthermore, we find form \eqref{bernfinal} that
\[
	\tilde{\Phi}_z-\tfrac1{\sqrt{3}} \tilde{\Phi} = N_3(\tilde{\Phi}, \tilde{\Psi})+\tilde{f}_3,
\]
where 
\[
	\tilde{f}_3 = N_3(U^{\star},V^{\star}).
\]
Applying Shauder estimates as in Proposition \ref{C2gamma} we conclude with
\begin{proposition}\label{C2gammatilde} For any $\gamma \in (0,1/2)$ and $\beta \in (0,2)$ there exist $C>0$ and $q_0>0$ such that
	\[
	\|\tilde{\Phi}\|_{C^{2,\gamma}([q,q+1]\times [0,\pitwo])} \leq C e^{-\beta q}
	\]
	for all $q \geq q_0$.
\end{proposition}

In fact, we can specify the asymptotics for $\Phi$ even further. For that purpose one needs to solve \eqref{sys:Phidag}-\eqref{sys:Psidag}, where $N_j(q)(\tilde{\Phi},\tilde{\Psi})$ are set to zero and $\tilde{f}_j$ contain only terms of order $e^{-(5/2)q}$, where $\xi$ is replaced by its approximation
\[
	\tilde{\xi} = \frac{\omega(1)}{3}e^{-\tfrac12 q},
\]
which is obtained from \eqref{xiphi}. After a long and tedious calculation one solves reduced equations for $\tilde{\Phi}$ and finds that
\begin{equation}\label{xiasymp}
\xi(q) = \frac{\omega(1)}{3}e^{-\tfrac12 q} + \lambda \omega^2(1) e^{-q} + O(e^{-\tfrac32(\tau_1-1)q}).
\end{equation}
The constant $\lambda \approx 1.1869$ is found numerically, though it can also be done analytically. Furthermore, similar asymptotics are valid for $\xi_q$ and $\xi_{qq}$ and can be obtained by differentiating the latter formula.

\section{Asymptotics for $\eta$ and $\psi$}\label{secasymppsi}

Based on \eqref{asympPhi} we can obtain the corresponding expansions for $\eta_x$ and $\psi$. First, using \eqref{xiasymp} and \eqref{extazeta} we find
\begin{equation}\label{etaasymptmp}
	\begin{split}
		\eta_x & = \frac{\sin(-(\pi/6)+\xi) - \xi_t \cos(-(\pi/6)+\xi)}{\cos(-(\pi/6)+\xi) + \xi_t \sin(-(\pi/6)+\xi)} \\
		& =-\frac{1}{\sqrt{3}}+ \frac{4}{3} \left(\xi-\xi_t\right)-\frac{4 \left(-2 \xi \xi_t+\xi_t^2+\xi^2\right)}{3 \sqrt{3}} + O(\xi^3+\xi_t^3) \\
		& = -\frac{1}{\sqrt{3}}+\frac23 \omega(1)e^{-\tfrac12 t}-\frac{\sqrt{3}-24\lambda}{9} \omega^2(1) e^{-t} + O\left(e^{-\tfrac32(\tau_1-1)t}\right).
	\end{split}
\end{equation}
Note that
\begin{equation}\label{expt}
	e^{-t} = \sqrt{x^2+(r-\eta(x))} = \frac{2}{3^{\tfrac12}} x (1  + O(x^{\tfrac12})),
\end{equation}
so that
\[
	\eta_x = -\frac{1}{\sqrt{3}} +  \kappa \sqrt{x} + O(x), \ \ \kappa = \frac{2^{\tfrac32}}{3^{\tfrac54}} \omega(1).
\]
We can use this information in order to specify the error term in \eqref{expt}, which gives
\[
	e^{-t} =  \frac{2}{3^{\tfrac12}} x \left(1- \frac{\kappa}{2\sqrt{3}} \sqrt{x} + O(x)\right), \ \	e^{-\tfrac12 t} =  \frac{2^{\tfrac12}}{3^{\tfrac14}} \sqrt{x} \left(1- \frac{\kappa}{4\sqrt{3}} \sqrt{x} + O(x)\right).
\]
Finally, using this in \eqref{etaasymptmp} we conclude
\[
	\eta_x = -\frac{1}{\sqrt{3}} +  \kappa \sqrt{x} - \left\{ \frac{1}{18}+\frac{2}{\sqrt{3}} \frac{\sqrt{3}-24\lambda}{9}  \right\} \omega^2(1) x + O\left(x^{-\tfrac32(\tau_1-1)}\right)
\]
as $x \to 0+$. Note that the coefficient
\[
- \left\{ \frac{1}{18}+\frac{2}{\sqrt{3}} \frac{\sqrt{3}-24\lambda}{9}  \right\} \approx 3.37
\]
is positive.

In a similar way one obtains asymptotics for the stream function $\psi$:
\[
	\psi(x,y) = 1 - \cos{\tfrac32 (\theta+\pitwo)} (x^2 + (r-y)^2)^{3/2} + O(x^2+y^2),
\]
where the next order terms can be found explicitly, although the formulas are more complicated.


\bibliographystyle{siam}
\bibliography{bibliography}
\end{document}